\documentclass[reqno]{amsart}

\usepackage{amsmath}
\usepackage{amsfonts}
\usepackage{amssymb}
\usepackage{amsthm,url,color}

\theoremstyle{plain}
\newtheorem{thm}{Theorem}
\newtheorem{lem}[thm]{Lemma}

\newtheorem{cor}[thm]{Corollary}
\newtheorem{rem}[thm]{Remark}
\theoremstyle{definition}
\newtheorem{defn}[thm]{Definition}

\newcommand{\la}{\langle}
\newcommand{\ra}{\rangle}
\newcommand{\x}{\times}
\newcommand{\ox}{\otimes}
\newcommand{\surj}{\twoheadrightarrow}

\newcommand{\cM}{{\mathcal{M}}}

\newcommand{\cR}{{\mathcal{R}}}

\newcommand{\PP}{\mathbb{P}}
\newcommand{\ZZ}{\mathbb{Z}}

\newcommand{\CC}{\mathbb{C}}
\newcommand{\RR}{\mathbb{R}}

\newcommand{\FF}{\mathbb{F}}

\newcommand{\Gr}{\operatorname{Gr}}
\newcommand{\Spec}{\operatorname{Spec}}

\newcommand{\GL}{\operatorname{GL}}
\newcommand{\SL}{\operatorname{SL}}

\newcommand{\PGL}{\operatorname{PGL}}

\newcommand{\ba}{\mathbf{a}}
\newcommand{\bb}{\mathbf{b}}
\newcommand{\bc}{\mathbf{c}}
\newcommand{\bg}{\boldsymbol{\gamma}}
\newcommand{\bm}{\boldsymbol{\mu}}
\newcommand{\bl}{\boldsymbol{\lambda}}
\newcommand{\bz}{\mathbf{0}}
\newcommand{\quot}{/\!\!/}

\DeclareMathOperator{\tr}{tr}

\DeclareMathOperator{\Sym}{Sym}
\DeclareMathOperator{\Hom}{Hom}

\DeclareMathOperator{\Id}{Id}

\title[E-polynomial of character varieties]{E-polynomial 
of the $\SL(3,\CC)$-character variety of free groups}

\author[S. Lawton]{Sean Lawton}

\address{Department of Mathematical Sciences, George Mason University,
4400 University Drive, Fairfax, Virginia  22030, USA}

\email{slawton3@gmu.edu}

\author[V. Mu\~{n}oz]{Vicente Mu\~{n}oz}
\address{Facultad de Ciencias Matem\'aticas, Universidad 
Complutense de Madrid, Plaza de Ciencias 3, 28040 Madrid, Spain}
\email{vicente.munoz@mat.ucm.es}



\begin{document}

\begin{abstract}
We compute the $E$-polynomial of the character variety of representations of a rank $r$ free group in $\SL(3,\CC)$.  Expanding upon techniques developed in \cite{LMN}, we stratify the space of representations and compute the $E$-polynomial of each geometrically described stratum using fibrations.  Consequently, we also determine the $E$-polynomial of its smooth, singular, and abelian loci and the corresponding Euler characteristic in each case.  Along the way, we give a new proof of results in \cite{CL}.
\end{abstract}

\maketitle

\section{Introduction}

Let $\Gamma$ be a finitely generated group, and let $G$ be
a complex reductive algebraic group. The space of $G$-representations
is 
  $$
  \cR(\Gamma,G)=\{\rho:\Gamma \to G \, | \, \rho \text{ is a group morphism}\}.
  $$
Writing a presentation $\Gamma=\la x_1,\ldots,x_n | R_1,\ldots, R_s\ra$, we have
that $\rho\in \cR(\Gamma,G)$ is determined by the images $A_i= \rho(x_i)$, $1\leq i\leq n$.
Hence we can write $\rho=(A_1,\ldots,A_n)$. These matrices are subject to the
relations $R_j(A_1,\ldots,A_n)=\Id$, $1\leq j\leq s$. Hence
  $$
  \cR(\Gamma,G)\cong \{(A_1,\ldots,A_n)\in G^n \ |\ R_1(A_1,\ldots,A_n)=\ldots =
R_s(A_1,\ldots,A_n)=\Id\}
  $$
is an affine algebraic set, since $G$ is algebraic.  

There is an action of $G$ by conjugation on $\cR(\Gamma,G)$, which is equivalent to the action of $PG=G/Z(G)$, where $Z(G)$ is the
center of $G$, since the center acts trivially. The {\it $G$-character variety of $\Gamma$} is the GIT quotient
  $$
  \cM(\Gamma,G)=  \cR(\Gamma,G)\quot G,
  $$
which is an affine algebraic set by construction. Note that if we write $X:=\cR(\Gamma,G)=\Spec(S)$,
then $X\quot G=\Spec(S^G)$. 

Every element $g\in \Gamma$ determines a character $\chi_g:X\to \CC$, $\chi_g(\rho)=
\tr(\rho(g))$, with respect to an embedding $G\hookrightarrow \GL(n,\CC)$. These regular functions $\chi_g\in S$ are invariant by conjugation, and
hence $\chi_g\in S^G$. Consider the algebra of characters 
  $$
  T =\CC[\chi_g \, |\, g\in \Gamma]\subset S^G,
 $$ 
and let $\chi(\Gamma,G)=\Spec(T)$. There is a
well-defined surjective map $\cM(\Gamma,G) \to \chi(\Gamma,G)$, which is an isomorphism when $G=\SL(n,\CC)$ among other examples; see \cite{Si}.

In this paper we are interested in the character variety for the free group on $r$ elements
$\Gamma=F_r$ and for the group $G=\SL(3,\CC)$. We compute the $E$-polynomial
(also known as Hodge-Deligne polynomial) of $\cM(F_r,\SL(3,\CC))$. The $E$-polynomial of
$\cM(F_r,\SL(2,\CC))$ has been computed in \cite{CL} by arithmetic methods (using the 
Weil conjectures). Recently, in \cite{MR}, the $E$-polynomials of $\cM(F_r, \PGL(n,\CC))$ 
have also been computed by arithmetic methods, where the result is given in the
form of a generating function.

Here we use a geometric technique, introduced in \cite{LMN}, to compute
$E$-polynomials of character varieties. This consists of stratifying the space of representations
geometrically, and computing the $E$-polynomials of each stratum using the behavior
of $E$-polynomials with fibrations. This technique is used in \cite{LMN} for the case
of $\Gamma=\pi_1(X)$ for a surface $X$ of genus $g=1,2$ and
$G=\SL(2,\CC)$ (and also with one puncture, fixing the holonomy around the puncture). 
The case of $g=3$ is worked out in \cite{MM}, the case of $g\geq 4$ in
\cite{MM2}, and the case of $g=1$ with two punctures 
appears in \cite{LM}. To implement this geometric technique for character varieties 
for $\SL(n,\CC)$, for $n\geq3$, 
we need to introduce the \emph{equivariant 
Hodge-Deligne polynomial} with respect to a finite group action on an affine variety. This
will be useful for studying character varieties of surface groups in 
$\SL(n,\CC)$, $n\geq 3$.

We start by recovering the $E$-polynomials  $e(\cM(F_r,\SL(2,\CC)))$ of \cite{CL} and
$e(\cM(F_r,\PGL(2,\CC)))$ of \cite{MR}, verifying that they are equal. Then we move 
to rank $3$ to compute $e(\cM(F_r,\SL(3,\CC)))$ and $e(\cM(F_r,\PGL(3,\CC)))$. They 
turn out to be equal again. 
The latter one coincides, as expected, with the polynomial obtained in \cite{MR}.

Unlike the methods used in \cite{MR} to obtain $e(\cM(F_r,\PGL(3,\CC)))$, our method provides an explicit geometric description of, and the $E$-polynomial for, each stratum.  By results in \cite{FL2} this additional information determines the $E$-polynomial of the smooth and singular loci of $\cM(F_r,\SL(3,\CC))$, and by \cite{FL3} also determines the $E$-polynomial of the abelian character variety $\cM(\mathbb{Z}^r,\SL(3,\CC))$.

Our main theorem is thus:

\begin{thm}\label{thm:main}
The $E$-polynomials $e(\cM(F_r,\SL(3,\CC)))=e(\cM(F_r,\PGL(3,\CC)))$ and they are equal to
  \begin{align*}
& (q^8-q^6-q^5+q^3)^{r-1}+
(q-1)^{2r-2} (q^{3r-3} -q^{r}) + 
\frac16 (q-1)^{2r-2} q(q+1)+
 \frac12 (q^2-1)^{r-1}q(q-1)  \\ & + 
 \frac13 (q^2+q+1)^{r-1} q(q + 1) - (q-1)^{r-1} q^{r-1} ( q^2-1)^{r-1} (2 q^{2r-2}-q).
 \end{align*}
\end{thm}

From the definition of the $E$-polynomial of a variety $X$, the classical Euler characteristic is given by $\chi(X)=e(X;1,1)$.  Consequently, we deduce: 

\begin{cor} 
Let $r\geq 2$.
The Euler characteristic of $\cM(F_r,\SL(3,\CC))$, $\cM(F_r,\PGL(3,\CC))$, and by \cite{FL} also $\cM(F_r,\mathrm{SU}(3))$, is given by $2 \cdot 3^{r-2}$.  The Euler characteristic of $\cM(\mathbb{Z}^r,\SL(3,\CC))$, and by \cite{FL3} also $\cM(\mathbb{Z}^r,\mathrm{SU}(3))$, is given by $3^{r-2}$.
\end{cor}

\noindent \textbf{Acknowledgments.}
We are grateful to S. Mozgovoy and M. Reineke for providing us with a copy of \cite{MR}, and also for giving us an explicit formula for $e(\cM(F_r,\PGL(3,\CC)))$ for checking against our polynomials. We also thank the anonymous referee for helping improve the exposition of this article.  Lawton was supported by the Simons Foundation Collaboration grant 245642, and the U.S. NSF grant DMS 1309376.  Mu\~{n}oz was partially supported by Project MICINN (Spain) MTM2010-17389.  We also acknowledge support from U.S. NSF grants DMS 1107452, 1107263, 1107367 ``RNMS: GEometric structures And Representation varieties" (the GEAR Network).

\section{Hodge structures and $E$-polynomials} \label{sec:e-poly}

Our main goal is to compute the $E$-polynomial (Hodge-Deligne polynomial) of the $\SL(3,\CC)$-character variety 
of a free group. We will follow the methods in \cite{LMN}, so we collect some basic results from \cite{LMN} in this section.

We start by reviewing the definition of the Hodge-Deligne polynomial. 
A pure Hodge structure of weight $k$ consists of a finite dimensional complex vector space
$H$ with a real structure, and a decomposition $H=\bigoplus_{k=p+q} H^{p,q}$
such that $H^{q,p}=\overline{H^{p,q}}$, the bar meaning complex conjugation on $H$.
A Hodge structure of weight $k$ gives rise to the so-called Hodge filtration, which is a descending filtration
$F^{p}=\bigoplus_{s\ge p}H^{s,k-s}$. We define $\Gr^{p}_{F}(H):=F^{p}/ F^{p+1}=H^{p,k-p}$.

A mixed Hodge structure consists of a finite dimensional complex vector space $H$ with a real structure,
an ascending (weight) filtration $\cdots \subset W_{k-1}\subset W_k \subset \cdots \subset H$
(defined over $\RR$) and a descending (Hodge) filtration $F$ such that $F$ induces a pure Hodge structure of weight $k$ on each $\Gr^{W}_{k}(H)=W_{k}/W_{k-1}$. We define $H^{p,q}:= \Gr^{p}_{F}\Gr^{W}_{p+q}(H)$ and write $h^{p,q}$ for the {\em Hodge number} $h^{p,q} :=\dim H^{p,q}$.

Let $Z$ be any quasi-projective algebraic variety (possibly non-smooth or non-compact). 
The cohomology groups $H^k(Z)$ and the cohomology groups with compact support  
$H^k_c(Z)$ are endowed with mixed Hodge structures \cite{De,De2}. 
We define the {\em Hodge numbers} of $Z$ by
$h^{k,p,q}_{c}(Z)= h^{p,q}(H_{c}^k(Z))=\dim \Gr^{p}_{F}\Gr^{W}_{p+q}H^{k}_{c}(Z)$ .
The Hodge-Deligne polynomial, or $E$-polynomial is defined as 
 $$
 e(Z)=e(Z)(u,v):=\sum _{p,q,k} (-1)^{k}h^{k,p,q}_{c}(Z) u^{p}v^{q}.
 $$

The key property of Hodge-Deligne polynomials that permits their calculation is that they are additive for
stratifications of $Z$. If $Z$ is a complex algebraic variety and
$Z=\bigsqcup_{i=1}^{n}Z_{i}$, where all $Z_i$ are locally closed in $Z$, then 
 $$
 e(Z)=\sum_{i=1}^{n}e(Z_{i}).
 $$
Also, by \cite[Remark 2.5]{LMN}, if $G\to X\to B$ is a principal fiber bundle with $G$ a 
connected algebraic group, then $e(X)=e(G)e(B)$. In general
we shall use this as $e(X/G)=e(X)/e(G)$ when $B=X/G$. In particular, 
if $Z$ is a $G$-space, and there is a subspace $B\subset Z$ such that 
$B\x G \to Z$ is surjective and it is an $H$-homogeneous space for a connected subgroup
$H\subset G$, then
  \begin{equation}\label{eqn:1}
  e(Z)=e(B)e(G)/e(H).
  \end{equation}

\begin{defn}
Let $X$ be a complex quasi-projective variety on which a finite group $F$ acts. 
Then $F$ also acts of the cohomology $H^*_c(X)$ respecting the mixed Hodge structure. So
$[H^*_c(X)]\in R(F)$, the representation ring of $F$.  The \emph{equivariant 
Hodge-Deligne polynomial} is defined as
 $$
 e_F(X)=\sum_{p,q,k}  (-1)^k [H^{k,p,q}_c(X)] \, u^pv^q \in R(F)[u,v].
 $$
\end{defn}

Note that the map $\dim:R(F)\to \ZZ$ gives $\dim(e_F(X))=e(X)$.

For instance, for an action of $\ZZ_2$, there are two irreducible
representations $T,N$, where $T$ is the trivial representation, and $N$ is the non-trivial representation.
Then $e_{\ZZ_2}(X)=aT+bN$. Clearly 
 \begin{align*}
  e(X) &= a+b , \\
  e(X/\ZZ_2) &= a .
 \end{align*}
In the notation of \cite[Section 2]{LMN}, $a=e(X)^+$, $b=e(X)^-$. 
Note that if $X,X'$ are spaces with $\ZZ_2$-actions, then 
writing $e_{\ZZ_2}(X)=aT+bN$, $e_{\ZZ_2}(X')=a'T+b'N$, we have
$e_{\ZZ_2}(X\x X')= (aa'+bb') T+ (ab'+ba')N$ and so
  \begin{equation}\label{eqn:2}
  e((X\x X')/\ZZ_2)=aa'+bb'= e(X)^+e(X')^++ e(X)^-e(X')^-.
 \end{equation}

When $h_c^{k,p,q}=0$ for $p\neq q$, the polynomial $e(Z)$ depends only on the product $uv$.
This will happen in all the cases that we shall investigate here. In this situation, it is
conventional to use the variable $q=uv$. If this happens, we say that the variety is {\it of balanced type}.
For instance, $e(\CC^n)=q^n$.

\section{$E$-polynomial of the $\SL(2,\CC$)-character variety of free groups}\label{sl2}

Let $F_r$ denote the free group on $r$ generators. Then the space of
representations of $F_r$ in the group $\SL(2,\CC)$ is
 $$
 \cR_{r,2}=\Hom(F_r, \SL(2,\CC)) =\{(A_1,\ldots, A_r) \, | \, A_i \in \SL(2,\CC) \}= \SL(2,\CC)^r.
 $$
The group $\PGL(2,\CC)$ acts on $\cR_{r,2}$ by simultaneous conjugation of all matrices,
and the character variety is defined as the GIT quotient
 $$
 \cM_{r,2}= \cR_{r,2} \quot \PGL(2,\CC).
 $$
 
We aim to compute the $E$-polynomial of $\cM_{r,2}$ using the methods developed in \cite{LMN}
and to recover the results of \cite{CL}.
We have the following sets:
 \begin{itemize}
  \item Reducible representations $\cR^{red}_{r,2}\subset \cR_{r,2}$ and the corresponding set $\cM^{red}_{r,2} 
\subset \cM_{r,2}$
  of characters of reducible representations. A representation $\rho=(A_1,\ldots, A_r)$ is 
  reducible if and only if all $A_i$ share at least one eigenvector.
  \item Irreducible representations $\cR^{irr}_{r,2}\subset \cR_{r,2}$ and the corresponding set $\cM^{irr}_{r,2}
\subset \cM_{r,2}$
  of characters of irreducible representations. This is the complement of $\cR^{red}_{r,2}$. It consists
  of the representations $\rho$ such that $\PGL(2,\CC)$ acts freely on $\rho$, and
  the orbit $\PGL(2,\CC)\cdot \rho$ is closed. Therefore $\cM^{irr}_{r,2}=\cR^{irr}_{r,2}/\PGL(2,\CC)$.
 \end{itemize}

\subsection{The reducible locus} \label{sec:3.1}
Let us start by computing $e(\cM^{red}_{r,2})$. For a reducible representation, we have a basis 
of $\CC^2$ in which
 $$
 \rho =\left( \left(\begin{array}{cc} \lambda_1 & * \\ 0 & \lambda_1^{-1} \end{array}\right), 
\left(\begin{array}{cc} \lambda_2 & * \\ 0 & \lambda_2^{-1} \end{array}\right), 
 \ldots,
\left(\begin{array}{cc} \lambda_r & * \\ 0 & \lambda_r^{-1} \end{array} \right) \right).
 $$
The corresponding point is determined by $(\lambda_1,\ldots,\lambda_r) \in (\CC^*)^r$,
modulo $(\lambda_1,\ldots,\lambda_r)  \sim (\lambda_1^{-1},\ldots,\lambda_r^{-1})$.
Note that the action of $\lambda \mapsto \lambda^{-1}$ on $X=\CC^*$ has $e(X)^+=q$
and $e(X)^-=-1$. Writing $X_i=\CC^*$, $i=1,\ldots, r$, we have that
 \begin{align*}
 e(X_1\x \ldots \x X_r)^+ &=  \sum_{\epsilon\in A} 
 \prod_{i=1}^r e(X_i)^{\epsilon_i}  \\
 &= q^r+ \binom{r}{2} q^{r-2} + \binom{r}{4} q^{r-4} + \ldots + \binom{r}{2[\frac{r}2]}
q^{r-2[\frac{r}2]} \\
 &= \frac12 ( (q+1)^r+ (q-1)^r),
 \end{align*}
where $A=  \{ (\epsilon_1,\ldots,\epsilon_r)  \in (\pm 1)^r \, | \, \prod \epsilon_i=+1\}$.
Also
 \begin{align*}
 e(X_1\x \ldots \x X_r)^- &=   e(X_1\x \ldots \x X_r) - 
 e(X_1\x \ldots \x X_r)^+ \\ 
 &=  (q-1)^r- \frac12 ( (q+1)^r+ (q-1)^r) \\
 &= \frac12 ( (q-1)^r- (q+1)^r).
 \end{align*}
Also note that $e(\cM^{red}_{r,2})= e((X_1\x \ldots \x X_r)/\ZZ_2)= e(X_1\x \ldots \x X_r)^+$.

\subsection{The reducible representations} \label{sec:3.2}
Now we move to the computation of $e(\cR^{red}_{r,2})$. We stratify the space as follows 
$\cR^{red}_{r,2}=R_0\cup R_1\cup R_2\cup R_3$, where:

\begin{itemize}
\item $R_0$ consists of $(A_1,\ldots,A_r)=(\pm \Id,\ldots, \pm \Id)$. So $e(R_0)=2^r$.
\item $R_1$ consists of 
 $$
 \rho \sim \left( \left(\begin{array}{cc} \lambda_1 & 0 \\ 0 & \lambda_1^{-1} \end{array}\right), 
\left(\begin{array}{cc} \lambda_2 & 0 \\ 0 & \lambda_2^{-1} \end{array}\right), 
 \ldots,
\left(\begin{array}{cc} \lambda_r & 0 \\ 0 & \lambda_r^{-1} \end{array} \right) \right),
 $$
that is, abelian representations (all matrices are diagonalizable with respect to the same basis).
Here $(\lambda_1,\ldots,\lambda_r) \neq (\pm 1,\ldots, \pm 1)$. 
Therefore this space is parametrized by 
 $$
 (\PGL(2,\CC)/D \x \left( (\CC^*)^r -\{ (\pm 1,\ldots, \pm 1)\} \right) )/\ZZ_2 ,
 $$
where $D$ is the space of diagonal matrices.
We know that $e (\PGL(2,\CC)/D )^+=q^2$, $e (\PGL(2,\CC)/D )^-=q$ by 
\cite[Proposition 3.2]{LMN}. 
By our computation above, for $B=(\CC^*)^r -\{ (\pm 1,\ldots, \pm 1)\}$, 
we have $e(B)^+= \frac12 ( (q+1)^r + (q-1)^r)-2^r$ and  
$e(B)^- = \frac12 ( (q-1)^r - (q+1)^r)$.
Therefore
 \begin{align*}
  e(R_1) &=  e (\PGL(2,\CC)/D )^+ e(B)^+ +e (\PGL(2,\CC)/D )^- e(B)^- \\
 &=  q^2  \frac12 ( (q+1)^r + (q-1)^r-2^r) + q  \frac12 ( (q-1)^r- (q+1)^r) \\
 &= \frac12 (q^2-q) (q+1)^r + \frac12  (q^2+q) (q-1)^r -q^2 2^r\, .
 \end{align*}

\item $R_2$ consists of 
 $$
 \rho \sim \left( \left(\begin{array}{cc} \pm 1& a_1 \\ 0 & \pm 1 \end{array}\right), 
\left(\begin{array}{cc} \pm 1& a_2 \\ 0 & \pm 1 \end{array}\right), 
 \ldots,
\left(\begin{array}{cc} \pm 1& a_r \\ 0 & \pm 1 \end{array}\right)\right),
 $$
where $(a_1,\ldots,a_r)\in \CC^r-\{0\}$. 
Let $B_2$ be the space of representations as above with respect to the canonical basis.
Therefore, there is a canonical surjective map $B_2\x \PGL(2,\CC) \surj R_2$.
The fibers of this map are given $H_2=
\left\{\left(\begin{array}{cc} a& b \\ 0 & a^{-1} \end{array}\right)\right\}\cong \CC^*\x \CC$.
That is, $H_2\to B_2\x \PGL(2,\CC) \to R_2$ is a fibration to which we apply 
Formula (\ref{eqn:1}) to obtain:
 $$
 e(R_2)= \frac{e(B_2)e(\PGL(2,\CC))}{e(H_2)}
 =\frac{2^r (q^r-1)(q^3-q)}{q(q-1)}= 2^r(q^r-1)(q+1) .
 $$

\item $R_3$ consists of 
 $$
 \rho \sim \left( \left(\begin{array}{cc} \lambda_1 & b_1 \\ 0 & \lambda_1^{-1} \end{array}\right), 
\left(\begin{array}{cc} \lambda_2 & b_2 \\ 0 & \lambda_2^{-1} \end{array}\right), 
 \ldots,
\left(\begin{array}{cc} \lambda_r & b_r \\ 0 & \lambda_r^{-1} \end{array} \right) \right),
 $$
where $\lambda_i\in \CC^*$, $(\lambda_1,\ldots,\lambda_r)\neq (\pm 1,\ldots, \pm 1)$.
Here,  $(b_1,\ldots, b_r)\in \CC^r$ and the upper diagonal matrices
$\left(\begin{array}{cc} 1& y \\ 0 & 1 \end{array}\right)$
transforms
$(b_1,\ldots, b_r) \mapsto (b_1+ y(\lambda_1-\lambda_1^{-1}),\ldots, b_r+ y(\lambda_r-\lambda_r^{-1}))$.
As  $(\lambda_1,\ldots,\lambda_r)\neq (\pm 1,\ldots, \pm 1)$, this action is non-trivial.
Note that  $(b_1,\ldots, b_r)$ does not live in the line spanned by 
$(\lambda_1-\lambda_1^{-1},\ldots, \lambda_r-\lambda_r^{-1})$.
There is a fibration $H_3\to B_3\x \PGL(2,\CC) \to R_3$ where
$H_3= \left\{\left(\begin{array}{cc} a& b \\ 0 & a^{-1} \end{array}\right)\right\}\cong \CC^*\x \CC$.
Thus
  \begin{align*}
 e(R_3) &=  (q^{r}-q) ((q-1)^r - 2^r)  e(\PGL(2,\CC))/q(q-1)\\
 &=  \frac{q^{r-1}-1}{q-1} ((q-1)^r - 2^r)  (q^3-q)\\
 &=  (q^{r-1}-1)(q-1)^{r-1} (q^3-q) - 2^r \frac{q^{r-1}-1}{q-1}(q^3-q) .
 \end{align*}
\end{itemize}

Now we add all the subsets together:
  \begin{align*}
  e(\cR^{red}_{r,2}) =& e(R_0)+e(R_1)+e(R_2)+e(R_3) \\
 =& \frac12 (q^2-q) (q+1)^r +\frac12  (q^2+q) (q-1)^r
 +  (q^{r-1}-1)(q-1)^{r-1} (q^3-q) .
 \end{align*}

\subsection{The irreducible locus}
Recall that $\cR^{irr}_{r,2}=\SL(2,\CC)^r-\cR^{red}_{r,2}$, so
 $$ 
 e(\cR^{irr}_{r,2}) = (q^3-q)^r-\frac12 (q^2-q) (q+1)^r - \frac12  (q^2+q) (q-1)^r
 -  (q^{r-1}-1)(q-1)^{r-1} (q^3-q) .
 $$
 So
 $$ 
 e(\cM^{irr}_{r,2}) =\frac{e(\cR^{irr}_{r,2})}{q^3-q}= (q^3-q)^{r-1}
 -\frac12 (q+1)^{r-1} - \frac12  (q-1)^{r-1}
 -  (q^{r-1}-1)(q-1)^{r-1} .
 $$
Finally, 
  \begin{align*}
e(\cM_{r,2}) &= e(\cM^{irr}_{r,2}) + e(\cM^{red}_{r,2}) = e(\cM^{irr}_{r,2}) +\frac12 ( (q+1)^r+ (q-1)^r) \\
 &=(q^3-q)^{r-1}
 +\frac12 q(q+1)^{r-1} + \frac12  q (q-1)^{r-1}
 -  q^{r-1}(q-1)^{r-1} .
 \end{align*}

This agrees with \cite{CL}.

\section{$E$-polynomial of the $\PGL(2,\CC)$-character variety of free groups}\label{sec:PGL2}

Let us compute the $E$-polynomial of $\cM(F_r,\PGL(2,\CC))$. The space of representations
will be denoted
 $$
\bar\cR_{r,2}=\Hom(F_r,\PGL(2,\CC))=\{(A_1,\ldots, A_r) | A_i\in \PGL(2,\CC)\}=\PGL(2,\CC)^r.
 $$
Note that $\PGL(2,\CC)=\SL(2,\CC)/\{\pm \Id\}$, so $\bar\cR_{r,2}=\cR_{r,2}/\{(\pm\Id,\ldots, \pm \Id)\}$.
The character variety is
 $$
 \bar\cM_{r,2}=\bar\cR_{r,2}\quot\PGL(2,\CC).
 $$

We denote by $\bar\cR_{r,2}^{red}$ and $\bar\cR_{r,2}^{irr}$ the subsets of 
reducible and irreducible representations, respectively, of $\bar\cR_{r,2}$. We denote by
$\bar\cM_{r,2}^{red}$ and $\bar\cM_{r,2}^{irr}$ the corresponding spaces in $\bar\cM_{r,2}$.

\subsection*{The reducible locus}
We first compute $e(\bar\cM^{red}_{r,2})$. A reducible representation in $\bar\cM^{red}_{r,2}$
is determined by the eigenvalues $(\lambda_1,\ldots,\lambda_r) \in (\CC^*)^r$,
modulo $\lambda_i\sim -\lambda_i$, $1\leq i\leq r$, and
$(\lambda_1,\ldots,\lambda_r)  \sim (\lambda_1^{-1},\ldots,\lambda_r^{-1})$.
So it is determined by  $(\lambda_1^2,\ldots,\lambda_r^2) \in (\CC^*)^r$,
modulo $(\lambda_1^2,\ldots,\lambda_r^2)  \sim (\lambda_1^{-2},\ldots,\lambda_r^{-2})$.
The space isomorphic to that in Section \ref{sec:3.1}, so
$e(\bar\cM_{r,2}^{red})= \frac12 ( (q+1)^r+ (q-1)^r)$.

\subsection*{The reducible representations}
Now we compute $e(\bar\cR^{red}_{r,2})$. We stratify it as
$\bar\cR^{red}_{r,2}=\bar R_0\cup \bar R_1\cup \bar R_2\cup \bar R_3$, where:

\begin{itemize}
\item $\bar R_0$ consists of one point $(A_1,\ldots,A_r)=(\Id,\ldots, \Id)$. So $e(R_0)=1$.
\item $\bar R_1$ consists of 
 $$
 \rho \sim \left( \left(\begin{array}{cc} \lambda_1 & 0 \\ 0 & \lambda_1^{-1} \end{array}\right), 
\left(\begin{array}{cc} \lambda_2 & 0 \\ 0 & \lambda_2^{-1} \end{array}\right), 
 \ldots,
\left(\begin{array}{cc} \lambda_r & 0 \\ 0 & \lambda_r^{-1} \end{array} \right) \right),
 $$
where the eigenvalues are determined by $(\lambda_1^2,\ldots,\lambda_r^2) \neq (1,\ldots, 1)$. 
This space is parametrized by 
$(\PGL(2,\CC)/D \x \left( (\CC^*)^r -\{ (1,\ldots, 1)\} \right) )/\ZZ_2$,
where $D$ is the space of diagonal matrices.
Using that $e (\PGL(2,\CC)/D )^+=q^2$, $e (\PGL(2,\CC)/D )^-=q$, and
$e(B)^+= \frac12 ( (q+1)^r + (q-1)^r)-1$, 
$e(B)^- = \frac12 ( (q-1)^r - (q+1)^r)$, for $B=((\CC^*)^r -\{ (1,\ldots, 1)\})$, 
we have 
 \begin{align*}
  e(\bar R_1) &=  e (\PGL(2,\CC)/D )^+ e(B)^+ +e (\PGL(2,\CC)/D )^- e(B)^- \\
 &=  q^2  \frac12 ( (q+1)^r + (q-1)^r-1) + q  \frac12 ( (q-1)^r- (q+1)^r) \\
 &= \frac12 (q^2-q) (q+1)^r + \frac12  (q^2+q) (q-1)^r -q^2 \, .
 \end{align*}
 
\item $\bar R_2$ consists of 
 $$
 \rho \sim \left( \left(\begin{array}{cc} 1& a_1 \\ 0 & 1 \end{array}\right), 
\left(\begin{array}{cc} 1& a_2 \\ 0 & 1 \end{array}\right), 
 \ldots,
\left(\begin{array}{cc} 1& a_r \\ 0 & 1 \end{array}\right)\right),
 $$
where $(a_1,\ldots,a_r)\in \CC^r-\{0\}$. Then
 $$
 e(\bar R_2)=e(R_2)/2^r= (q^r-1)(q+1) .
 $$

\item $\bar R_3$ consists of 
 $$
 \rho \sim \left( \left(\begin{array}{cc} \lambda_1 & b_1 \\ 0 & \lambda_1^{-1} \end{array}\right), 
\left(\begin{array}{cc} \lambda_2 & b_2 \\ 0 & \lambda_2^{-1} \end{array}\right), 
 \ldots,
\left(\begin{array}{cc} \lambda_r & b_r \\ 0 & \lambda_r^{-1} \end{array} \right) \right),
 $$
where $\lambda_i\in \CC^*$, $(\lambda_1^2,\ldots,\lambda_r^2)\neq (1,\ldots, 1)$.
Here  $(b_1,\ldots, b_r)\in \CC^r- \la (\lambda_1-\lambda_1^{-1},\ldots, \lambda_r-\lambda_r^{-1}) \ra$.
There is a fibration $H_3\to B_3\x \PGL(2,\CC) \to \bar R_3$ where
$B_3$ parametrizes $(\lambda_1^2,\ldots,\lambda_r^2)$ and $(b_1,\ldots,b_r)$, and
$H_3= \left\{\left(\begin{array}{cc} a& b \\ 0 & a^{-1} \end{array}\right)\right\}\cong \CC^*\x \CC$.
Then
  \begin{align*}
 e(\bar R_3) &=  (q^{r}-q) ((q-1)^r - 1)  e(\PGL(2,\CC))/q(q-1)\\
 &=  \frac{q^{r-1}-1}{q-1} ((q-1)^r - 1)  (q^3-q). 
 \end{align*}
\end{itemize}

Now we add all subsets together to obtain:
  \begin{align*}
  e(\bar\cR^{red}_{r,2}) =& e(\bar R_0)+e(\bar R_1)+e(\bar R_2)+e(\bar R_3) \\
 =& \frac12 (q^2-q) (q+1)^r +\frac12  (q^2+q) (q-1)^r +  (q^{r-1}-1)(q-1)^{r-1} (q^3-q) \\
 =&e(\cR^{red}_{r,2}).
 \end{align*}

\subsection*{The irreducible locus}
Clearly, as $e(\SL(2,\CC))=q^3-q=e(\PGL(2,\CC))$ and $e(\bar\cR^{red}_{r,2}) =e(\cR^{red}_{r,2})$,
we have that $e(\bar\cR^{irr}_{r,2})=e(\cR^{irr}_{r,2})$. Therefore
$e(\bar\cM^{irr}_{r,2})=e(\cM^{irr}_{r,2})$. Finally, since
$e(\bar\cM^{red}_{r,2})=e(\cM^{red}_{r,2})$, we have
that 
  \begin{align*}
 e(\bar\cM_{r,2}) &=e(\cM_{r,2}) 
 =(q^3-q)^{r-1} +\frac12 q(q+1)^{r-1} + \frac12  q (q-1)^{r-1} -  q^{r-1}(q-1)^{r-1} .
 \end{align*}

\section{$E$-polynomial of the $\SL(3,\CC$)-character variety for $F_1$}\label{sl3r1}

Having given a new geometric derivation of the $E$-polynomial for $\cM_{r,2}$ and $\bar\cM_{r,2}$, in the next sections we work out the $E$-polynomial of $\cM_{r,3}$ and $\bar\cM_{r,3}$ in a similar fashion.

However, in this section we first address the $r=1$ case.  Although it is easy to see that $\cM_{r,n}\cong \CC^{n-1}$ via the coefficients of the characteristic polynomial, and hence $e(\cM_{r,n})=q^{n-1}$, this case will motivate the more complicated stratification, and the use of the {\it equivariant} $E$-polynomial, needed to compute the general $E$-polynomials for $\cM_{r,3}$ and $\bar\cM_{r,3}$ when $r\geq 2$.

We begin by working out the $E$-polynomials for $\GL(3,\CC), \SL(3,\CC)$, and $\PGL(3,\CC)$.  Like in the previous sections, we then stratify $\Hom(F_1,\SL(3,\CC))$ by orbit-type and compute the $E$-polynomial for each strata.  

\begin{lem}\label{lem:PGL3}
 $e(\SL(3,\CC))=e(\PGL(3,\CC))=(q^3-1)(q^3-q)q^2= q^8-q^6-q^5+q^3$.
\end{lem}

\begin{proof}
 Consider $\CC^n$, and let $V_k$ be the Stiefel manifold of $k$ linearly independent
 vectors in $\CC^n$. Then, there is a (Zariski locally trivial) fibration $\CC^n-\CC^{k-1} \to V_k\to V_{k-1}$.
 Therefore $e(V_k)=\prod_{i=0}^{k-1} (q^n-q^i)$. So $e(\GL(n,\CC))=e(V_n)=\prod_{i=0}^{n-1}(q^n-q^i)$.
 
 Now there is a (Zariski locally trivial) fibration $\CC^* \to \GL(n,\CC) \to \PGL(n,\CC)$, hence
  $e(\PGL(n,\CC))=e(\GL(n,\CC))/(q-1)= q^{n-1}\prod_{i=0}^{n-2}(q^n-q^i)$.
  
  For $\SL(n,\CC)$, the choice of $(v_1,\ldots, v_{n-1})\in V_{n-1}$ determines an affine hyperplane
  $$\{v\in \CC^n | \det(v_1,\ldots, v_{n-1},v)=1\}.$$ This gives a (Zariski locally trivial) affine
  bundle $\CC^{n-1} \to \SL(n,\CC) \to V_{n-1}$, and hence 
  $e(\SL(n,\CC))= q^{n-1}\prod_{i=0}^{n-2}(q^n-q^i)$.
 \end{proof}

Now let us consider the representations of $F_1$ to $\SL(3,\CC)$. This is equivalent to studying the conjugation action 
of $\PGL(3,\CC)$ on $X:=\SL(3,\CC)$. For this action, there are $6$ strata types.  In the following list, we write down all 6 strata, but include the computation of their $E$-polynomials for only the first 5.  This is because the computation is apparent from the geometric description of each stratum alone in those cases.  Here they are:

\begin{itemize}
 \item $X_0$ formed by matrices of type $\left( \begin{array}{ccc} \xi & 0 & 0 \\ 0 &\xi & 0\\
0 & 0 &\xi \end{array} \right)$.
Here $\xi^3=1$. So $X_0$ consists of $3$ points and $e(X_0)=3$.

\item  $X_1$ formed by matrices of type $\left( \begin{array}{ccc} \xi & 0 & 0 \\ 0 &\xi & 1\\
0 & 0 &\xi \end{array} \right)$.
Here $\xi^3=1$, so $\xi$ admits $3$ values. The stabilizer of
this matrix is $U_1=\left\{ \left( \begin{array}{ccc} \mu^{-2} & 0 & b \\ a &\mu & c\\
0 & 0 &\mu \end{array} \right) \right\} \cong \CC^* \x \CC^3$. So 
$e(X_1)=3 e(\PGL(3,\CC)/U_1)=3(q^3-1)(q^3-q)q^2/q^3(q-1) = 3q^4+3q^3-3q-3$.

\item $X_2$ formed by matrices of type $\left( \begin{array}{ccc} \xi & 1 & 0 \\ 0 &\xi & 1 \\
0 & 0 &\xi \end{array} \right)$.
Here $\xi^3=1$, so $\xi$ admits $3$ values. The stabilizer of
this matrix is $U_2=\left\{ \left(\begin{array}{ccc} 1 & b & c \\ 0 &1 & b\\
0 & 0 &1 \end{array} \right) \right\} \cong \CC^2$. So 
$e(X_1)=3 e(\PGL(3,\CC)/U_2)=3(q^3-1)(q^3-q)q^2/q^2 = 3q^6-3q^4-3q^3+3q$.

\item $X_3$ formed by matrices of type 
$\left( \begin{array}{ccc} \lambda & 0 & 0 \\ 0 &\lambda & 0 \\
0 & 0 &\lambda^{-2} \end{array} \right)$, where $\lambda \in \CC^* -\{\xi | \xi^3=1\}$.
The stabilizer of this matrix is $U_3=\left\{ 
\left(\begin{array}{cc} A & 0 \\ 0 & (\det A)^{-1}\end{array} \right)
\ |\ A\in \GL(2,\CC)\right\} \cong \GL(2,\CC)$. So 
$e(X_3)=(q-4) e(\PGL(3,\CC)/U_3)=(q-4)(q^3-1)(q^3-q)q^2/(q^2-1)(q^2-q) 
=q^5-3q^4-3q^3-4q^2$.

\item $X_4$ formed by matrices of type $\left( \begin{array}{ccc} \lambda & 1 & 0 \\ 0 &\lambda & 0 \\
0 & 0 &\lambda^{-2} \end{array} \right)$, where $\lambda \in \CC^* -\{\xi \ |\ \xi^3=1\}$.
The stabilizer of this matrix is $U_4=\left\{\left( \begin{array}{ccc} \mu & b & 0 \\ 0 &\mu & 0\\
0 & 0 &\mu^{-2} \end{array} \right) \right\} \cong \CC^* \x\CC$. So 
$e(X_4)=(q-4) e(\PGL(3,\CC)/U_4)=(q-4)(q^3-1)(q^3-q)q^2/q(q-1) 
= q^7-3q^6-4q^5-q^4+3q^3+4q^2$.

\item $X_5$ formed by matrices of type $\left( \begin{array}{ccc} \lambda & 0 & 0 \\ 0 &\mu & 0 \\
0 & 0 &\gamma \end{array} \right)$, where $\lambda,\mu,\gamma \in \CC^*$ are different and $\lambda\mu\gamma=1$. The
stabilizer is isomorphic to the diagonal matrices $D\cong \CC^* \x \CC^*$. The parameter space
is 
 $$ 
 B=\{(\lambda,\mu) \in (\CC^*)^2\ |\ \lambda \neq \mu^{-2}, 
 \mu \neq \lambda^{-2},\mu \neq \lambda\}.
 $$
The map $\PGL(3,\CC)/D \x B \to X_5$ is a 6:1 cover. Moreover,
  $$
   X_5 \cong (\PGL(3,\CC)/D \x B)/\Sigma_3,
  $$
where the symmetric group $\Sigma_3$ acts on $\PGL(3,\CC)$ by permuting the columns
and acts on the triple $(\lambda,\mu,\gamma=\lambda^{-1}\mu^{-1})$ by permuting the entries.
\end{itemize}

We now compute $e(X_5)$ using the {\it equivariant} $E$-polynomial. Consider the finite group $F=\Sigma_3$. The representation ring $R(F)$ is generated by
three irreducible representations:
 \begin{itemize}
 \item $T$ is the (one-dimensional) trivial representation.
 \item $S$ is the sign representation. This is one-dimensional and given by
 the sign map $\Sigma_3\to \{\pm 1\} \subset \GL(1,\CC)$.
 \item $V$ is the two-dimensional representation given as follows. Take $St=\CC^3$ the standard 
$3$-dimensional representation. This is generated by $e_1,e_2,e_3$ and $\Sigma_3$ acts by
permuting the elements of the basis. Then $T=\la e_1+e_2+e_3\ra$ and we can decompose
$St=T\oplus V$.
 \end{itemize}
The representation ring $R(\Sigma_3)$ has a multiplicative structure given by:
$T\ox T=T$, $T\ox S=S$, $T\ox V=V$, $S\ox S=T$, $S\ox V=V$, $V\ox V=T\oplus S\oplus V$.

\begin{lem} \label{lem:ld}
\begin{itemize}
\item[]
 \item $e_{\Sigma_3}(B)=(q^2-q+1) T +  S - 2(q-2) V$.
 \item $e_{\Sigma_3}(\PGL(3,\CC)/D) =q^6 T + q^3 S+ (q^5+q^4)V$.
\end{itemize}
\end{lem}

\begin{proof}
Write, for a quasi-projective variety $X$ with a $\Sigma_3$-action, $e_{\Sigma_3}(X)=a T+bS +cV$.
Then $a=e(X/\Sigma_3)$. If we consider the cycle $(1,2)$ and the subgroup $H=\la (1,2)\ra$,
there is a map $R(F)\to R(H)$ which sends $T\mapsto T$, $S\mapsto N$ and
$V\mapsto T+N$. Then $e_H(X)=aT+bN+c(T+N)=(a+c)T+(b+c)N$. 
Therefore, $a+c=e(X/H)$. As $e(X)=a+b+2c$, we
can compute $a,b,c$ by knowing these $E$-polynomials.

For  $B=\{(\lambda,\mu) \in (\CC^*)^2\ |\ \lambda \neq \mu^{-2}, 
\mu \neq \lambda^{-2} ,\mu\neq \lambda\}$,
the three curves $\lambda =\mu^{-2}$, $\mu =\lambda^{-2}$, $\mu=\lambda$
intersect at the three points
$\{(\xi,\xi)\ |\ \xi^3=1\}$. Hence $e(B)=(q-1)^2-3(q-4)-3 =q^2-5q+10$.

Now $\Sigma_3$ acts on $(\lambda,\mu,\gamma)$ and the quotient space is
parametrized by $s=\lambda+\mu+\gamma$, $t=\lambda\mu+\lambda\gamma+\mu\gamma$
and $p=\lambda\mu\gamma=1$, that is, by $(s,t)\in \CC^2$. We have to remove the cases
$s=\lambda+\lambda^{-2}+\lambda$, $t=\lambda^{-1}+\lambda^2+\lambda^{-1}$.
This defines a rational curve in $\CC^2$. It has two points at infinity. The map $\lambda \mapsto
(2\lambda+\lambda^{-2},2\lambda^{-1}+\lambda^2)$ is an embedding. Therefore
$e(B/\Sigma_3)=q^2-(q-1)=q^2-q+1$.

The action by $H$ permutes $(\lambda,\mu)$, hence the quotient is parametrized by
$s'=\lambda+\mu$, $p'=\lambda\mu\neq 0$. We have to remove the cases
$s'=\lambda+\lambda^{-2}$, $p'=\lambda^{-1}$, that is, $s'=(p')^{-1}+(p')^2$; 
and $s'=2\lambda$, $p'= \lambda^2$, i.e., $4p'=(s')^2$. They intersect at three points. Then
$e(B/H)=q(q-1)-2(q-1)+3=q^2-3q+5$.

Thus
 $$
 e_{\Sigma_3}(B)=(q^2-q+1) T +  S - 2(q-2) V. 
 $$

For $C=\PGL(3,\CC)/D$, the space $C$ consists of points in $(\PP^2)^3-\Delta$, where
$\Delta$ is the diagonal (triples of coplanar points). 
Certainly, a matrix in $\GL(3,\CC)$ can be written as
$(v_1,v_2,v_3)$, where $v_1,v_2,v_3$ are linearly independent vectors. Taking a quotient by
the diagonal matrices corresponds to the vectors up to a scalar: $[v_1],[v_2],[v_3]$.
Therefore, $e(C)=(q^3-1)(q^3-q)q^2/(q-1)^2=q^6+2q^5+2q^4+q^3$.

The group $\Sigma_3$ acts by permuting the vectors, so $C/\Sigma_3=\Sym^3\PP^2- \bar\Delta$, where
$\bar\Delta$ consists of linearly dependent triples $([v_1],[v_2],[v_3])$. If they are equal, 
the set has $e(\PP^2)=q^2+q+1$. If they are collinear, there is a fibration with fiber
$\Sym^3(\PP^1)-\Delta$ and base $(\PP^2)^{\vee}$. This has $E$-polynomial $(1+q+q^2+q^3-1-q)(1+q+q^2)=
q^5+2q^4+2q^3+q^2$. Also
$e(\Sym^3\PP^2)=q^6+q^5+2q^4+2q^3+2q^2+q+1$.
Therefore
  $$
e(C/\Sigma_3)=q^6+q^5+2q^4+2q^3+2q^2+q+1-(q^5+2q^4+2q^3+q^2 +q^2+q+1)=q^6.
 $$

The group $H$ acts by permuting the first two vectors, so $C/H=\Sym^2\PP^2
\x\PP^2- \bar\Delta'$, where $\bar\Delta'$ consists of linearly dependent triples $([v_1],[v_2],[v_3])$.
If $[v_1]=[v_2]$, we have $E$-polynomial $(q^2+q+1)(q^2+q+1)=q^4+2q^3+3q^2+2q+1$.
If $[v_1]\neq [v_2]$, they lie in $\Sym^2\PP^2-\Delta$ and we have $E$-polynomial
$(q^4+q^3+2q^2+q+1-(q^2+q+1))(q+1)=q^5+2q^4+2q^3+q^2$.
Also $e(\Sym^2\PP^2\x\PP^2)=(q^4+q^3+2q^2+q+1)(q^2+q+1)$, so
  \begin{align*}
e(C/H) =& q^6+2q^5+4q^4+  4q^3+4q^2+2q+1 \\
  &-(q^5+2q^4+2q^3+q^2+q^4+2q^3+3q^2+2q+1) \\
 =& q^6+q^5+q^4.
 \end{align*}

This produces the polynomial
 $$
e_{\Sigma_3}(C)=q^6 T + q^3 S+ (q^5+q^4)V.
 $$
\end{proof}

\begin{rem}\label{rem:ld}
 If we consider $B'=\{(\lambda,\mu) \in (\CC^*)^2\}$, then the proof of Lemma \ref{lem:ld}
 says that $e_{\Sigma_3}(B')= q^2T+S-qV$.
\end{rem}

Now suppose $e_{\Sigma_3}(X)=aT+bS+cV$ and $e_{\Sigma_3}(X')=a'T+b'S+c'V$. Then
 $$
 e_{\Sigma_3}(X\x X')=(aa'+bb'+cc')T + (ab'+ba'+cc') S+ (ac'+ca'+bc'+cb'+cc')V,
 $$
and hence
 \begin{equation}\label{eqn:ta}
 e((X\x X')/\Sigma_3)=aa'+bb'+cc' \, .
 \end{equation}
We finally obtain the $E$-polynomial for the sixth strata $X_5$:
 \begin{align*}
 e(X_5) &=e((B\x C)/\Sigma_3 ) \\ 
 &= (q^2-q+1) q^6  +q^3  - 2(q-2) (q^5+q^4) \\
 &= q^8-q^7  -q^6 +2q^5+4q^4+q^3 \, .
\end{align*}

Now we add the strata together:
 $$
 e(X_0)+e(X_1)+e(X_2)+e(X_3)+e(X_4)+e(X_5)=q^8-q^6-q^5+q^3 = e(\SL(3,\CC)),
 $$
as expected.
 
\begin{rem}
All elements of $X=\SL(3,\CC)$ are reducible. The semisimple ones are given
by diagonal matrices with entries $\lambda,\mu,\gamma$ with $\lambda\mu\gamma=1$.
So they are parametrized by $s=\lambda + \mu + \gamma$, $t=\lambda\mu+\lambda\gamma+
\mu\gamma=\lambda^{-1}+\gamma^{-1}+\mu^{-1}$, for $(s,t)\in \CC^2$. Hence
$e(\cM_{1,3})=q^2$, as noted at the beginning of this section.
\end{rem}

\section{$E$-polynomials of character varieties for $F_r$, $r>1$, and $\SL(3,\CC)$}\label{sec:SL3}

In this section we prove (most of) our main theorem (Theorem \ref{thm:main}) by computing the $E$-polynomial for $\cM_{r,3}$; the rest of Theorem \ref{thm:main} is proved in Section \ref{pgl3}. The computation is similar to the computation in Section \ref{sl2} except the stratification is more complicated and the {\it equivariant} $E$-polynomial is needed, as was demonstrated in Section \ref{sl3r1}.

Indeed, we want to study the space of representations
 \begin{align*}
\cR_{r,3} &= \Hom(F_r,\SL(3,\CC))=\{\rho:F_r\to \SL(3,\CC)\} = \{(A_1,\ldots, A_r) | A_i\in \SL(3,\CC) \}=\SL(3,\CC)^r
 \end{align*}
and the corresponding character variety
 $$
 \cM_{r,3}=\Hom(F_r,\SL(3,\CC))\quot\PGL(3,\CC).
 $$

Much of the algebraic structure of $\cM_{r,3}$ has been worked out in \cite{La1, La2, La3}. 
 
Let us start by computing the $E$-polynomial of the space of reducible representations
$\cR_{r,3}^{red} \subset  \Hom(F_r,\SL(3,\CC))$. 

We now list the stratification and the computation of the $E$-polynomial for each stratum for $\cR_{r,3}^{red}$:

\begin{enumerate}

\item $R_0=R_{01}\cup R_{02}$. $R_{01}$ is 
formed by representations $\rho=(A_1,\ldots, A_r)$ which have a common eigenvector
and such that the quotient representation is irreducible, that is, 
 $$
 A_i=\left(\begin{array}{cc} \lambda_i^{-2} & b_i\, \,\,  c_i \\
 {0 \atop 0} &\lambda_i B_i\end{array}\right),
 $$ 
where $(B_1,\ldots, B_r) \in \cR^{irr}_{2,r}$. Let $B_{01}$ be the space of representations of
such form with respect to the standard basis. The stabilizer of $B_{01}$ 
(i.e. the set $H_{01}\subset \PGL(3,\CC)$
sending $B_{01}$ to itself) is $H_{01}= \left\{ \left( \begin{array}{cc} (\det B)^{-1} & a \,\, \, b \\ {0\atop 0} 
& B\end{array} \right) \right\} \cong \GL(2,\CC)\x \CC^2$. This means that there is a fibration
$H_{01} \to B_{01}\x \PGL(3,\CC)\to R_{01}$. Hence 
  $$ 
 e(R_{01})=(q-1)^r q^{2r} e(\cR^{irr}_{2,r}) \frac{e(\PGL(3,\CC))}{q^2e(\GL(2,\CC))}.
  $$

$R_{02}$ is 
formed by representations $\rho=(A_1,\ldots, A_r)$ which have a common two-dimensional space
and upon which it acts irreducibly, that is, 
 $$
 A_i=\left(\begin{array}{cc}  \lambda_i B_i & {0 \atop 0}  \\
b_i\, \,\,  c_i & \lambda_i^{-2}  \end{array}\right),
 $$ 
where $(B_1,\ldots, B_r) \in \cR^{irr}_{2,r}$. The stabilizer is now
 $H_{02}= \left\{ \left( \begin{array}{cc} B & {0\atop 0} \\
a \,\, \, b & (\det B)^{-1} \end{array} \right) \right\} \cong \GL(2,\CC)\x \CC^2$. Hence 
  $$ 
 e(R_{02})=(q-1)^r q^{2r} e(\cR^{irr}_{2,r}) \frac{e(\PGL(3,\CC))}{q^2e(\GL(2,\CC))}.
  $$
Finally, the intersection $R_{01}\cap R_{02}$ consists of those representations with
$b_i=c_i=0$, which have stabilizer $\GL(2,\CC)$, hence
  $$ 
 e(R_{01}\cap R_{02})=(q-1)^r  e(\cR^{irr}_{2,r}) \frac{e(\PGL(3,\CC))}{e(\GL(2,\CC))}.
  $$

Finally $e(R_0)=e(R_{01})+e(R_{02})- e(R_{01}\cap R_{02})=2e(R_{01})- e(R_{01}\cap R_{02})$.
Note that the remaining representations have a full invariant flag.

\item $R_1$ is formed by representations $\rho=(A_1,\ldots, A_r)$ such that
the eigenvalues of all $A_i$ are equal (and hence cubic roots of unity). This consists of the following
substrata:

\begin{itemize}
\item $R_{11}$ consisting of matrices
$A_i=\left( \begin{array}{ccc} \xi_i & 0 & 0 \\ 0 &\xi_i & 0\\
0 & 0 &\xi_i \end{array} \right)$, where $\xi_i^3=1$. So $e(R_{11})=3^r$.

\item  $R_{12}$ formed by matrices of type 
$A_i=\left( \begin{array}{ccc} \xi_i & 0 & 0 \\ 0 &\xi_i & a_i\\
0 & 0 &\xi_i \end{array} \right)$, with $\xi_i^3=1$ and
$(a_1,\ldots,a_r)\neq 0$. Then the stabilizer is
$H_{12}= \left\{ \left( \begin{array}{ccc} \mu^{-1}\gamma^{-1} & 0 & b \\ a &\mu & c\\
0 & 0 &\gamma \end{array} \right) \right\} \cong (\CC^*)^2 \x \CC^3$. So 
 $$
 e(R_{12})=3^r (q^r-1) \frac{e(\PGL(3,\CC))}{(q-1)^2q^3}.
 $$

\item $R_{13}$ formed by matrices of type 
$A_i=\left( \begin{array}{ccc} \xi_i & 0 & a_i \\ 0 &\xi_i & b_i\\
0 & 0 &\xi_i \end{array} \right)$, with $\xi_i^3=1$ with
$(a_1,\ldots,a_r), (b_1,\ldots, b_r)$ linearly independent. Note that when
they are linearly dependent, one may arrange a basis so that it
belongs to the stratum $R_{12}$. Then the stabilizer is
$H_{13}=
\left\{ \left( \begin{array}{cc} A& {b\atop c} \\ 0\,\, \, 0 & (\det A)^{-1} \end{array} \right)\right\} \cong 
\GL(2,\CC) \x \CC^2$, hence
 $$
 e(R_{13})=3^r (q^r-1)(q^r-q) \frac{e(\PGL(3,\CC))}{(q^2-1)(q^2-q)q^2}.
 $$

\item $R_{14}$ formed by matrices of type 
$A_i=\left( \begin{array}{ccc} \xi_i & a_i & b_i \\ 0 &\xi_i & 0\\
0 & 0 &\xi_i \end{array} \right)$, with $\xi_i^3=1$ and
$(a_1,\ldots,a_r), (b_1,\ldots, b_r)$ linearly independent. Note again that when
they are linearly dependent, one may arrange a basis so that it
belongs to the stratum $R_{12}$. Then the stabilizer is
$H_{14}=
\left\{ \left( \begin{array}{cc}(\det A)^{-1} & b\ \ c\\{0\atop 0}& A \end{array} \right)\right\} \cong 
\GL(2,\CC) \x \CC^2$, hence
 $$
 e(R_{14})=3^r (q^r-1)(q^r-q) \frac{e(\PGL(3,\CC))}{(q^2-1)(q^2-q)q^2}.
 $$

\item $R_{15}$ formed by matrices of type 
$A_i=\left( \begin{array}{ccc} \xi_i & a_i & b_i \\ 0 &\xi_i & c_i\\
0 & 0 &\xi_i \end{array} \right)$, with $\xi_i^3=1$ and
$(a_1,\ldots, a_r),(c_1,\ldots,c_r)$ are both non-zero (if one of them is
zero, then we are back in the case $R_{13}$).
Then  the stabilizer is
$H_{15}=\left\{ \left( \begin{array}{ccc} a & b & c \\ 0 &d & e\\
0 & 0 & f \end{array} \right) \right\}$. Hence
 $$
 e(R_{14})=3^r (q^r-1)^2q^r \frac{e(\PGL(3,\CC))}{(q-1)^2q^3}.
 $$
\end{itemize}
All together, we have
 $$
 e(R_1)=3^r \big( 1 + (1 + q + q^2) (q^{3 r+1}+ q^{3 r}- 2 q^{2 r+1}+q-1) \big).
 $$

\item  $R_{2}$ formed by matrices with eigenvalues $(\lambda_i,\lambda_i,\mu_i)$. 
Let $\bl=(\lambda_1,\ldots, \lambda_r), \bm=(\mu_1,\ldots, \mu_r)$, with
 $\bl-\bm\neq \bz$. Note that $\mu_i=\lambda_i^{-2}$,
so the parameter space has $E$-polynomial $(q-1)^r-3^r$.
We have the substrata:

\begin{itemize}
\item $R_{21}$ consists of representations of type 
$A_i=\left( \begin{array}{ccc} \lambda_i & 0 & 0 \\ 0 &\lambda_i & 0\\
0 & 0 &\mu_i \end{array} \right)$. The stabilizer
is $P(\GL(2,\CC)\x \CC^*)\cong \GL(2,\CC)$, so 
 $$
 e(R_{21})=((q-1)^r-3^r)\frac{e(\PGL(3,\CC))}{(q^2-1)(q^2-q)}.
 $$

\item $R_{22}$ consists of representations of type
$A_i=\left( \begin{array}{ccc} \lambda_i & a_i & b_i \\ 0 &\lambda_i & c_i\\
0 & 0 &\mu_i \end{array} \right)$, with $\ba\neq \bz$,
 $\bb=(b_1,\ldots,b_r), \bc=(c_1,\ldots,c_r) \in \CC^r$.
The stabilizer is $H_{22}=\left\{\left( \begin{array}{ccc} a& b& c\\ 0 &d& e\\
0 & 0 &f\end{array} \right)\right\}$. Hence
 $$
e(R_{22})=((q-1)^r-3^r) (q^r-1) q^{2r} \frac{e(\PGL(3,\CC))}{(q-1)^2q^3}.
 $$

\item $R_{23}$ consists of representations of type
$A_i=\left( \begin{array}{ccc} \lambda_i & 0 & 0 \\ 0 &\lambda_i & a_i\\
0 & 0 &\mu_i \end{array} \right)$, with $\ba \notin \la \bl-\bm\ra$. If
$\ba=x(\bl-\bm)$, $x\in \CC$, then we can arrange a basis so that this belongs
to the stratum $R_{21}$.
The stabilizer is
$H_{23}=\left\{\left( \begin{array}{ccc} a & 0 & 0 \\ b & c& d\\
0 & 0 & e \end{array} \right)\right\}$. So
 $$ 
 e(R_{23})=((q-1)^r-3^r)(q^r-q) \frac{e(\PGL(3,\CC))}{(q-1)^2q^2}.
 $$

\item $R_{24}$ consists of representations of type
$A_i=\left( \begin{array}{ccc} \lambda_i & 0 & a_i \\ 0 &\lambda_i & b_i\\
0 & 0 &\mu_i \end{array} \right)$, with $\ba, \bb$ and $\bl-\bm$ linearly independent 
(if they where linearly dependent, one can arrange a basis so that we go back
to case $R_{23}$). The stabilizer is
$H_{24}=\left\{\left( \begin{array}{cc} A & {a \atop b} \\ 
0\, \, \, 0 & (\det A)^{-1} \end{array} \right)\right\}$. Hence
 $$
e(R_{24})=((q-1)^r-3^r) (q^r-q)(q^r-q^2) \frac{e(\PGL(3,\CC))}{(q^2-1)(q^2-q)q^2}.
 $$

\item $R_{25}$ consists of representations of type
$A_i=\left( \begin{array}{ccc} \mu_i & a_i & b_i \\ 0 &\lambda_i & c_i\\
0 & 0 &\lambda_i \end{array} \right)$, with $\ba \notin \la \bl-\bm\ra$, $\bc \neq \bz$.
The stabilizer is $H_{25}=\left\{\left( \begin{array}{ccc} a& b& c\\ 0 &d& e\\
0 & 0 &f\end{array} \right)\right\}$. Hence
 $$
e(R_{25})=((q-1)^r-3^r) (q^r-1)(q^r-q) q^r\frac{e(\PGL(3,\CC))}{(q-1)^2q^3}.
 $$

\item $R_{26}$ consists of representations of type
$A_i=\left( \begin{array}{ccc} \mu_i & 0 & b_i \\ 0 &\lambda_i & c_i\\
0 & 0 &\lambda_i \end{array} \right)$, with $\bb \notin \la \bl-\bm\ra$, $\bc \neq \bz$.
(If $\bb$ is a multiple of $\bl-\bm$, then we can arrange with a suitable basis that
$\bb=0$, and this belongs to the substrata $R_{22}$).
The stabilizer is $H_{26}=\left\{\left( \begin{array}{ccc} a& 0& b\\ 0 &c& d\\
0 & 0 &e\end{array} \right)\right\}$. Hence
 $$
e(R_{26})=((q-1)^r-3^r) (q^r-1)(q^r-q) \frac{e(\PGL(3,\CC))}{(q-1)^2q^2}.
 $$

\item $R_{27}$ consists of representations of type
$A_i=\left( \begin{array}{ccc} \mu_i & a_i & 0 \\ 0 &\lambda_i & 0\\
0 & 0 &\lambda_i \end{array} \right)$, with $\ba \notin \la \bl-\bm\ra$. 
The stabilizer is $H_{27}=\left\{\left( \begin{array}{ccc} a& b& 0\\ 0 &c& 0\\
0 & d &e\end{array} \right)\right\}$. Hence
 $$
e(R_{27})=((q-1)^r-3^r) (q^r-q) \frac{e(\PGL(3,\CC))}{(q-1)^2q^2}.
 $$

\item $R_{28}$ consists of representations of type
$A_i=\left( \begin{array}{ccc} \mu_i & a_i & b_i \\ 0 &\lambda_i & 0\\
0 & 0 &\lambda_i \end{array} \right)$, with $\ba,\bb, \bl-\bm$ linearly independent (otherwise
we can reduce to the case $R_{27}$). 
The stabilizer is $H_{28}=\left\{\left( \begin{array}{cc} (\det A)^{-1} & b \,\,\, c\\ {0\atop 0} & A\end{array} \right)\right\}$. Hence
 $$
e(R_{28})=((q-1)^r-3^r) (q^r-q)(q^r-q^2) \frac{e(\PGL(3,\CC))}{(q^2-q) (q^2-1)q^2}.
 $$

\item $R_{29}$ consists of representations of type
$A_i=\left( \begin{array}{ccc} \lambda_i & a_i & b_i \\ 0 &\mu_i & c_i\\
0 & 0 &\lambda_i \end{array} \right)$, with $\ba,\bc \notin \la\bl-\bm\ra$.
The stabilizer is $H_{29}=\left\{\left( \begin{array}{ccc} a& b& c\\ 0 &d& e\\
0 & 0 &f\end{array} \right)\right\}$. Hence
 $$
e(R_{29})=((q-1)^r-3^r) (q^r-q)^2q^r  \frac{e(\PGL(3,\CC))}{(q-1)^2q^3}.
 $$

\end{itemize}
All together, we have
 \begin{align*}
 e(R_2)= & ((q-1)^r-3^r) 
(q^2+q+1)(3 q^{3 r+1}+ 3 q^{3 r} - 2 q^{2 r+2} - 4 q^{ 2 r+1} +q^3).
 \end{align*}

\item  $R_{3}$ is formed by matrices with eigenvalues $(\lambda_i,\mu_i,\gamma_i)$ such that 
$\bl=(\lambda_1,\ldots,\lambda_r), 
\bm=(\mu_1,\ldots,\mu_r), 
\bg=(\gamma_1,\ldots,\gamma_r)$ satisfy $\bl-\bm, \bm-\bg, \bl-\bg \neq \bz$.
Note that $\lambda_i\mu_i\gamma_i=1$ for all $1\leq i\leq r$.
The
base $B_r$ parametrizing $(\bl,\bm,\bg)$ has $E$-polynomial
$e(B_r)=(q-1)^{2r}-3(q-1)^{r}+2 \cdot 3^r$.

\begin{itemize}
\item $R_{31}$ consists of representations of type
$A_i=\left( \begin{array}{ccc} \lambda_i & 0& 0 \\ 0 &\mu_i & 0\\
0 & 0 &\gamma_i \end{array} \right)$. Then the stabilizer is
 $D\x \Sigma_3$, where $D$ are the diagonal matrices. So
we have to compute the $E$-polynomial of the quotient 
$R_{31}=(\PGL(3,\CC)/D \times B_r)/\Sigma_3$.
We start by computing $e_{\Sigma_3}(B_r)$. Let $B_r'=\{(\bl,\bm,\bg) \in (\CC^*)^{3r}\ |\ \bl\bm\bg=(1,\ldots,1)\}$.
This is $B_r'=(B')^r$, in the notation of Remark \ref{rem:ld}. Then 
 $$
  e_{\Sigma_3}(B_r') = e_{\Sigma_3}(B')^r = (q^2 T+ S-qV)^r.
 $$
Using the properties 
$T\ox T=T$, $T\ox S=S$, $T\ox V=V$, $S\ox S=T$, $S\ox V=V$, $V\ox V=T\oplus S\oplus V$, 
it is easy to see that $V^b= a_b V + a_{b-1} (T+S)$, where $a_b=a_{b-1}+2a_{b-2}$, with
$a_0=0, a_1=1$. This recurrence solves as $a_b=(2^b-(-1)^b)/3$. Therefore:
 \begin{align}\label{eqn:Br'}
  e_{\Sigma_3}(B_r') =& \,  (q^2 T+ S-qV)^r\\
  =&  \,\sum \frac{r!}{(r-a-b)!a!b!} q^{2(r-a-b)} S^a (-q)^b V^b \nonumber \\
  =&  \,\sum \frac{r!}{(r-a)!a!} q^{2(r-a)} S^a +
  \sum_{b>0} \frac{r!}{(r-a-b)!a!b!} q^{2(r-a-b)} S^a (-q)^b V^b \nonumber \\
  =&  \, \frac{(q^2+1)^r+(q^2-1)^r}2 T +\frac{(q^2+1)^r-(q^2-1)^r}2 S \nonumber \\ 
   &+ \sum_{b>0} \frac{r!}{(r-a-b)!a!b!} q^{2(r-a-b)} S^a (-q)^b \left( \frac{2^b-(-1)^b}{3} V+ 
  \frac{2^{b-1}-(-1)^{b-1}}{3} (T+S) \right) \nonumber \\
  =&  \, \frac{(q^2+1)^r+(q^2-1)^r}2 T +\frac{(q^2+1)^r-(q^2-1)^r}2 S 
   + \frac13 \left( (q^2-2q+1)^r -(q^2+q+1)^r \right) V \nonumber \\
   &+ \frac13 \left( \frac12 (q^2-2q+1)^r + (q^2+q+1)^r -  \frac32 (q^2+1)^r \right) (T+S)\nonumber  \\
  =&  \, \left(\frac{(q^2-1)^r}2+\frac16 (q-1)^{2r} + \frac13 (q^2+q+1)^r\right) T  \nonumber \\
   & + \left(- \frac{(q^2-1)^r}2+\frac16 (q-1)^{2r} + \frac13 (q^2+q+1)^r\right) S 
   + \frac13 \left( (q-1)^{2r} -(q^2+q+1)^r \right) V . \nonumber
 \end{align}
 
Now we have to look at the part that we removed: $C_r=\{(\bl,\bl,\bl^{-2})\ |\ \bl\in (\CC^*)^r\}
\cup\{(\bl,\bl^{-2},\bl)\ |\ \bl\in (\CC^*)^r\} \cup\{(\bl^{-2},\bl,\bl)\ |\ \bl\in (\CC^*)^r\}$.
Then $e(C_r)=3(q-1)^r-2\cdot 3^r$. The quotient $C_r/\Sigma_3\cong (\CC^*)^r$, 
so $e(C_r/\Sigma_3)=(q-1)^r$. And for $H=\la (1,2)\ra$ we have $C_r/H \cong 
\{(\bl,\bl,\bl^{-2})\ |\ \bl\in (\CC^*)^r\} \cup\{(\bl,\bl^{-2},\bl)\ |\ \bl\in (\CC^*)^r\}$,
so $e(C_r/H)=2(q-1)^r - 3^r$. Hence,
 $$
  e_{\Sigma_3}(C_r)= (q-1)^r T+ ((q-1)^r-3^r)V.
  $$
For $B_r=B_r'-C_r$, we have 
 \begin{align}\label{eqn:S3}
 e_{\Sigma_3}(B_r)
 =&  \, \left(\frac{(q^2-1)^r}2+\frac16 (q-1)^{2r} + \frac13 (q^2+q+1)^r - (q-1)^r \right) T  \\
   & + \left(- \frac{(q^2-1)^r}2+\frac16 (q-1)^{2r} + \frac13 (q^2+q+1)^r\right) S \nonumber \\
   &+\left(  \frac13 (q-1)^{2r} - \frac13 (q^2+q+1)^r -(q-1)^r+3^r\right) V .\nonumber 
 \end{align}

Hence Formula (\ref{eqn:ta}) and Lemma \ref{lem:ld} imply
  \begin{align*}
e(R_{31}) &= aa'+bb'+cc' \\
 &=\left(\frac{(q^2-1)^r}2+\frac16 (q-1)^{2r} + \frac13 (q^2+q+1)^r - (q-1)^r \right) q^6  \\
   & + \left(- \frac{(q^2-1)^r}2+\frac16 (q-1)^{2r} + \frac13 (q^2+q+1)^r\right) q^3 \\
   &+ \left(\frac13  (q-1)^{2r} -\frac13 (q^2+q+1)^r -(q-1)^r+3^r\right) (q^5+q^4).
 \end{align*}

\item $R_{32}$ consists of representations of type
$A_i=\left( \begin{array}{ccc} \lambda_i & 0& 0 \\ 0 &\mu_i & a_i\\
0 & 0 &\gamma_i \end{array} \right)$ with
$\ba\notin \la \bm-\bg \ra$. The stabilizer is 
$H_{32}=\left\{\left( \begin{array}{ccc} a& 0& 0 \\ 0 &b& c\\
0 & 0 &d \end{array} \right)\right\}$. Hence,
 $$
e(R_{32})=((q-1)^{2r}-3(q-1)^{r}+2\cdot 3^r) (q^r-q)\frac{e(\PGL(3,\CC))}{(q-1)^2q}.
 $$

\item $R_{33}$ consists of representations of type
$A_i=\left( \begin{array}{ccc} \lambda_i & 0& a_i \\ 0 &\mu_i & b_i\\
0 & 0 &\gamma_i \end{array} \right)$ with
$\ba\notin \la \bl-\bg \ra$, $\bb\notin \la \bm-\bg \ra$. The stabilizer is 
$H_{33}=\left\{\left( \begin{array}{ccc} a& 0& b \\ 0 &c& d\\
0 & 0 &e \end{array} \right)\right\} \x \ZZ_2$, where $\ZZ_2$ permutes
the eigenvalues $\lambda_i,\mu_i$. Therefore,
 $$
 R_{33}=\big( B_{r} \x (\CC^r-\CC)^2 \x (\PGL(3,\CC)/ H_{33}) \big)/\ZZ_2.
 $$
By (\ref{eqn:S3}), we have that
 \begin{align*}
 e_{H}(B_r)
 =\, & \left(\frac{(q^2-1)^r}2+\frac12 (q-1)^{2r}- 2(q-1)^r+3^r \right) T  \\
    &+ \left(- \frac{(q^2-1)^r}2+\frac12 (q-1)^{2r}- (q-1)^r+3^r \right) N ,
 \end{align*}
since under $H\subset \Sigma_3$, we have $T\mapsto T$, $S\mapsto N$, $V\mapsto T+N$.
For the second factor, $e((\CC^r-\CC)^2)=(q^r-q)^2$ and $e(\Sym^2(\CC^r-\CC))=q^{2r}-q^{r+1}$, so
 $$
 e_H( (\CC^r-\CC)^2)=(q^{2r}-q^{r+1}) T+ (q^2-q^{r+1}) N.
 $$
Finally, $\PGL(3,\CC)/ H_{33}\cong \PP^2\x \PP^2-\Delta$, by considering the first two columns of the matrix,
and where $\Delta$ is the diagonal. As $e(\PP^2\x \PP^2-\Delta)=(1+q+q^2)(q+q^2)$ and
$e(\Sym^2\PP^2-\bar\Delta)=q^4+q^3+q^2$, we have
 $$ 
 e_H(\PGL(3,\CC)/ H_{33})=(q^4+q^3+q^2)T+(q^3+q^2+q)N.
 $$
Hence,
 \begin{align*}
 e(R_{33}) = &\,  \left(\frac{(q^2-1)^r}2+\frac12 (q-1)^{2r}- 2(q-1)^r+3^r \right) 
 \big( (q^{2r}-q^{r+1})(q^4+q^3+q^2) + (q^2-q^{r+1})(q^3+q^2+q)\big) \\
 & +  \left(- \frac{(q^2-1)^r}2+\frac12 (q-1)^{2r}- (q-1)^r+3^r \right) 
 \big((q^{2r}-q^{r+1}) (q^3+q^2+q)+ (q^2-q^{r+1})(q^4+q^3+q^2)\big) .
\end{align*}

\item $R_{34}$ consists of representations of type
$A_i=\left( \begin{array}{ccc} \lambda_i & a_i& b_i \\ 0 &\mu_i & 0\\
0 & 0 &\gamma_i \end{array} \right)$ 
with $\ba\notin \la \bl-\bm \ra$, $\bb\notin \la \bl-\bg \ra$. The stabilizer is 
$H_{34}=\left\{\left( \begin{array}{ccc} a& b& c \\ 0 &d& 0\\
0 & 0 &e \end{array} \right)\right\}$. The computations are analogous to the case of $R_{33}$, so
$e(R_{33})=e(R_{34})$.

\item $R_{35}$ consists of representations of type
$A_i=\left( \begin{array}{ccc} \lambda_i & a_i& b_i \\ 0 &\mu_i & c_i\\
0 & 0 &\gamma_i \end{array} \right)$
with $\ba\notin \la \bl-\bm \ra$, $\bc\notin \la \bm-\bg \ra$. The stabilizer is 
$\left\{\left( \begin{array}{ccc} a& b& c \\ 0 &d& e\\
0 & 0 &f \end{array} \right)\right\}$. Hence
 $$
e(R_{35})=((q-1)^{2r}-3(q-1)^{r}+2\cdot
3^r) (q^r-q)^2q^r\frac{e(\PGL(3,\CC))}{(q-1)^2q^3}.
 $$
\end{itemize}
All together, we have:
 \begin{align*}
 e(R_3)= & (2\cdot 3^r - 3 (q-1)^r + (q-1)^{2 r})(q+1)(q^2+q+1)( q^r-q) (q^2 + q^{2 r} - q^{1 + r})  \\ 
  & +  (2 \cdot 3^r -  2 (q-1)^r + (q-1)^{2 r} - (q^2-1)^r) q (q^2+q+1) (q^r-q)(q^r-q^2) \\
  &+   (2 \cdot 3^r - 4 (q-1)^r + (q-1)^{2 r} + (q^2-1)^r) q^2 (q^2+q+1) ( q^r-1) (q^r-q)+  \\ 
  &+  \frac16 q^3 ((q-1)^{2 r} - 3 (q^2-1)^r + 2 (q^2+q+1)^r + 
    2 q (q+1) (3^{r+1 } - 3 ( q-1)^r + (q-1)^{2 r} - (q^2+q+1)^r) ) \\ 
& +   \frac16 q^6 (-6 (q-1)^r + (q-1)^{2 r} + 3 (q^2-1)^r + 2 (q^2+q+1)^r).
 \end{align*}
\end{enumerate}

Therefore,
  \begin{align*}
e(\cR^{red}_{r,3})= & \frac13 (q^2+q+1)^r (q-1 )^2 q^3 (q+1 ) +(q^2+q+1 ) (2q^{2r}-q^2) (q-1)^{2r} q^r (q+1)^r \\
 &-\frac13 ( q-1)^{2r}(q+1)(q^2+q+1)(3q^{3r}- 3q^{r+2}+ q^3),
 \end{align*}
and so,
 $$
 e(\cR^{irr}_{r,3})= e(\cR_{r,3})- e(\cR^{red}_{r,3}) = e(\SL(3,\CC))^r-e(\cR^{red}_{r,3}) ,
 $$
and consequently,
$$
e(\cM^{irr}_{r,3})=e(\cR^{irr}_{r,3})/e(\PGL(3,\CC))= e(\SL(3,\CC))^{r-1}-e(\cR^{red}_{r,3})/e(\SL(3,\CC)). 
 $$

\subsection*{$E$-polynomial of the moduli of reducible representations}
To compute $e(\cM_{r,3})$, it remains to compute the moduli space of reducible representations $\cM^{red}_{r,3}$.
This is formed by two strata:
 \begin{enumerate}
 \item $M_0$ formed by semisimple representations which split into irreducible representations of ranks 
$1$ and $2$, that is, of the form:
  $$
 A_i=\left(\begin{array}{cc} \lambda_i^{-2} & 0\, \,\,  0 \\
 {0 \atop 0} &\lambda_i B_i\end{array}\right),
 $$ 
where $(B_1,\ldots,B_r)\in \cM^{irr}_{r,2}$. So $e(M_0)=(q-1)^r e(\cM^{irr}_{r,2})$.
\item $M_1$ formed by semisimple representations which split into three irreducible representations of rank $1$.
These are given by eigenvalues $\bl=(\lambda_1,\ldots,\lambda_r), 
\bm=(\mu_1,\ldots,\mu_r), \bg=(\gamma_1,\ldots,\gamma_r) \in (\CC^*)^r$ where $\lambda_i\mu_i\gamma_i=1$
for all $1\leq i\leq r$. This is the space $B_r'$ whose $E$-polynomial has been computed in (\ref{eqn:Br'}). Thus
 $$
 e(M_1)=e(B_r'/\Sigma_3)= \frac{(q^2-1)^r}2+\frac16 (q-1)^{2r} + \frac13 (q^2+q+1)^r.
 $$
\end{enumerate}
Finally, $e(\cM^{red}_{r,3})=e(M_0)+e(M_1)$, and adding up everything we get
 \begin{align*}
 e(\cM_{r,3}) =& e(\cM^{irr}_{r,3})+e(\cM^{red}_{r,3}) \\
 =& (q^8-q^6-q^5+q^3)^{r-1}+
(q-1)^{2r-2} (q^{3r-3} -q^{r}) + 
\frac16 (q-1)^{2r-2} q(q+1)  \\ &+
 \frac12 (q^2-1)^{r-1}q(q-1)  + 
 \frac13 (q^2+q+1)^{r-1} q(q + 1) - (q-1)^{r-1} q^{r-1} ( q^2-1)^{r-1} (2 q^{2r-2}-q).
 \end{align*}

This completes the main part of the proof of Theorem \ref{thm:main}.  It remains to show that $e(\cM_{r,3})=e(\bar\cM_{r,3})$, which we do in the following section.

\begin{rem}
By \cite{FL2}, the singular locus of $\cM_{r,3}$ is exactly the reducible locus $($and so the smooth locus is its complement$)$.  Therefore, the above computation of $M_0$ and $M_1$ gives the $E$-polynomial of the singular locus of $\cM_{r,3}$.  Likewise, $e(\cM^{irr}_{r,3})$ is the $E$-polynomial of the smooth locus of $\cM_{r,3}$.  Moreover, by \cite{FL3}, the abelian character variety $\cM(\mathbb{Z}^r,\SL(3,\CC))$ is exactly the diagonalizable representations in $\cM_{r,3}$.  The above computation of $M_1$ gives the $E$-polynomial of $\cM(\mathbb{Z}^r,\SL(3,\CC))$.  In each case, setting $q=1$ gives the Euler characteristic of the corresponding space.
\end{rem}

\section{$E$-polynomials of character varieties for $F_r$, $r>1$, and $\PGL(3,\CC)$}\label{pgl3}

In this final section, we focus on the space of representations
 \begin{align*}
\bar\cR_{r,3} &= \Hom(F_r,\PGL(3,\CC))=\{\rho:F_r\to \PGL(3,\CC)\} 
 = \{(A_1,\ldots, A_r) | A_i\in \PGL(3,\CC) \}=\PGL(3,\CC)^r
 \end{align*}
and the character variety
 $$
 \bar\cM_{r,3}=\Hom(F_r,\PGL(3,\CC))\quot\PGL(3,\CC).
 $$
Let $\zeta=e^{2\pi \sqrt{-1}/3}$, and let $\ZZ_3=\{1,\zeta,\zeta^2\}$ be the space
of cubic roots of unity. Then  $\PGL(3,\CC)=\SL(3,\CC)/\ZZ_3$,
 \begin{align*}
  \bar\cR_{r,3} &= \cR_{r,3}/( \ZZ_3)^r, \text{ and }\\
  \bar\cM_{r,3} &= \cM_{r,3}/( \ZZ_3)^r 
 \end{align*}
where $(\zeta^{a_1},\ldots, \zeta^{a_r})$ acts as $(A_1,\ldots, A_r) \mapsto
(\zeta^{a_1}A_1,\ldots, \zeta^{a_r}A_r)$. Clearly $\bar\cR_{r,3}^{red}= \cR_{r,3}^{red}/( \ZZ_3)^r$
and $\bar\cR_{r,3}^{irr}= \cR_{r,3}^{irr}/( \ZZ_3)^r$.

We know from Lemma \ref{lem:PGL3} that $e(\PGL(3,\CC))=e(\SL(3,\CC))$. Let us see now that
$e(\bar\cR_{r,3}^{red})=e( \cR_{r,3}^{red})$. We stratify $\bar\cR_{r,3}^{red}=\bar R_0\sqcup \bar R_1
\sqcup \bar R_2 \sqcup \bar R_3$, where $\bar R_i=R_i /(\ZZ_3)^r$ and the $R_i$, $i=0,1,2,3$, have been
defined in Section \ref{sec:SL3}.  

We now list the strata with the computation of their $E$-polynomials:

\begin{enumerate}
\item $\bar R_0=\bar R_{01}\cup \bar R_{02}$, where $\bar R_{0j}=R_{0j}/(\ZZ_3)^r$, $j=1,2$.
To compute $e(\bar R_{01})$, recall that $R_{01}$ is formed by representations $\rho=(A_1,\ldots, A_r)$ with
 $$
 A_i=\left(\begin{array}{cc} \lambda_i^{-2} & b_i\, \,\,  c_i \\
 {0 \atop 0} &\lambda_i B_i\end{array}\right),
 $$ 
where $(B_1,\ldots, B_r) \in \cR^{irr}_{2,r}$. The action of $\zeta^{a_i}$ on $A_i$ is given by
$(\lambda_i, b_i,c_i, B_i)\mapsto (\zeta^{a_i}\lambda_i, \zeta^{a_i}b_i,\zeta^{a_i}c_i, \zeta^{a_i}B_i)$.
Note that $\CC/\ZZ_3 \cong \CC$ and $\CC^*/\ZZ_3 \cong \CC^*$, so the relevant cohomology is invariant.
Therefore 
 $$
 e(((\CC^*)^r\x \CC^r \x \CC^r \x \cR^{irr}_{2,r})/(\ZZ_3)^r)=  e(\CC^*)^r e(\CC)^r e(\CC)^r e(\cR^{irr}_{2,r}/(\ZZ_3)^r).
 $$ 
This means that $e(\bar R_{01})=e(R_{01})$. Analogously $e(\bar R_{02})=e(R_{02})$
and $e(\bar R_{01}\cap \bar R_{02})=e(R_{01}\cap R_{02})$, so $e(\bar R_{0})=e(R_{0})$.

\item $\bar R_1=R_1/(\ZZ_3)^r$. Note that $R_1$ is formed by $3^r$ copies of the same subvariety, hence
 $$
 e(\bar R_1)=\frac{e(R_1)}{3^r}= 1 + (1 + q + q^2) (q^{3 r+1}+ q^{3 r}- 2 q^{2 r+1}+q-1) .
 $$

\item  $\bar R_{2}=R_2/(\ZZ_3)^r$. Recall that $R_2$ is formed by matrices with eigenvalues $(\lambda_i,\lambda_i,\mu_i)$
where $\bl=(\lambda_1,\ldots, \lambda_r) \in P=(\CC^*)^r - \{1,\zeta,\zeta^2\}^r$. Now
$\bar P=P/(\ZZ_3)^r \cong (\CC^*)^r-\{(1,1,\ldots, 1)\}$, so $e(\bar P)=(q-1)^r-1$.
It is more or less straightforward to see that $\bar R_2$ can be stratified by $\bar R_{2j}=R_{2j}/(\ZZ_3)^r$, $j=1,2,\ldots, 9$.
For each $\bar R_{2j}$ the computation of $e(\bar R_{2j})$ is the same as that of $e(R_{2j})$ but replacing
$e(P)=(q-1)^r-3^r$ by $e(\bar P)=(q-1)^r-1$. Hence 
 \begin{align*}
 e(\bar R_2)= & ((q-1)^r-1)  (q^2+q+1)(3 q^{3 r+1}+ 3 q^{3 r} - 2 q^{2 r+2} - 4 q^{ 2 r+1} +q^3).
 \end{align*}

\item  $\bar R_{3}=R_3/(\ZZ_3)^r$. We follow the lines of the computation of $e(R_3)$. The base for the
space of eigenvalues is $\bar B_r=B_r/(\ZZ_3)^r$ with $e(\bar B_r)=(q-1)^{2r}-3(q-1)^{r}+2$.

\begin{itemize}
\item Let $\bar R_{31}=R_{31}/(\ZZ_3)^r \cong (\PGL(3,\CC)/D \times \bar B_r)/\Sigma_3$.
If $\bar B_r'=B_r'/(\ZZ_3)^r$, then easily 
$e_{\Sigma_3}(\bar B_r') = e_{\Sigma_3}(\bar B')^r = (q^2 T+ S-qV)^r =e_{\Sigma_3}(B_r')$.
For $\bar C_r=C_r/(\ZZ_3)^r$, we have instead that  $e_{\Sigma_3}(\bar C_r)= (q-1)^r T+ ((q-1)^r-1)V$,
so $\bar B_r=\bar B_r'-\bar C_r$ has 
 \begin{align*} 
 e_{\Sigma_3}(B_r)
 =&  \, \left(\frac{(q^2-1)^r}2+\frac16 (q-1)^{2r} + \frac13 (q^2+q+1)^r - (q-1)^r \right) T  \\
   & + \left(- \frac{(q^2-1)^r}2+\frac16 (q-1)^{2r} + \frac13 (q^2+q+1)^r\right) S \\
   &+\left(  \frac13 (q-1)^{2r} - \frac13 (q^2+q+1)^r -(q-1)^r+1\right) V ,
 \end{align*}
 and 
  \begin{align*}
e(\bar R_{31}) 
 &=\left(\frac{(q^2-1)^r}2+\frac16 (q-1)^{2r} + \frac13 (q^2+q+1)^r - (q-1)^r \right) q^6  \\
   & + \left(- \frac{(q^2-1)^r}2+\frac16 (q-1)^{2r} + \frac13 (q^2+q+1)^r\right) q^3 \\
   &+ \left(\frac13  (q-1)^{2r} -\frac13 (q^2+q+1)^r -(q-1)^r+1\right) (q^5+q^4).
 \end{align*}

\item $\bar R_{32}=R_{32}/(\ZZ_3)^r$ has 
 $$
e(\bar R_{32})=((q-1)^{2r}-3(q-1)^{r}+2) (q^r-q)\frac{e(\PGL(3,\CC))}{(q-1)^2q}.
 $$

\item $\bar R_{33}=R_{33}/(\ZZ_3)^r \cong \big( \bar B_{r} \x (\CC^r-\CC)^2 \x (\PGL(3,\CC)/ H_{33}) \big)/\ZZ_2$,
where $H=\ZZ_2$ acts by swapping the first two eigenvalues. Now
 \begin{align*}
 e_{H}(\bar B_r)
 =\, & \left(\frac{(q^2-1)^r}2+\frac12 (q-1)^{2r}- 2(q-1)^r+1 \right) T  \\
    &+ \left(- \frac{(q^2-1)^r}2+\frac12 (q-1)^{2r}- (q-1)^r+1 \right) N ,
 \end{align*}
so
 \begin{align*}
 e(\bar R_{33}) = &\,  \left(\frac{(q^2-1)^r}2+\frac12 (q-1)^{2r}- 2(q-1)^r+1 \right) 
 \big( (q^{2r}-q^{r+1})(q^4+q^3+q^2) + (q^2-q^{r+1})(q^3+q^2+q)\big) \\
 & +  \left(- \frac{(q^2-1)^r}2+\frac12 (q-1)^{2r}- (q-1)^r+1 \right) 
 \big((q^{2r}-q^{r+1}) (q^3+q^2+q)+ (q^2-q^{r+1})(q^4+q^3+q^2)\big) .
\end{align*}

\item $\bar R_{34}=R_{34}/(\ZZ_3)^3$  has $e(\bar R_{34})=e(\bar R_{33})$.
\item $\bar R_{35}=R_{35}/(\ZZ_3)^r$ has
 $$
e(\bar R_{35})=((q-1)^{2r}-3(q-1)^{r}+2) (q^r-q)^2q^r\frac{e(\PGL(3,\CC))}{(q-1)^2q^3}.
 $$
\end{itemize}

All together, we have:
 \begin{align*}
 e(\bar R_3)= & (2 - 3 (q-1)^r + (q-1)^{2 r})(q+1)(q^2+q+1)( q^r-q) (q^2 + q^{2 r} - q^{ r+1})  \\ 
  & +  (2 -  2 (q-1)^r + (q-1)^{2 r} - (q^2-1)^r) q (q^2+q+1) (q^r-q)(q^r-q^2) \\
  &+   (2  - 4 (q-1)^r + (q-1)^{2 r} + (q^2-1)^r) q^2 (q^2+q+1) ( q^r-1) (q^r-q)+  \\ 
  &+  \frac16 q^3 ((q-1)^{2 r} - 3 (q^2-1)^r + 2 (q^2+q+1)^r + 
    2 q (q+1) (3  - 3 ( q-1)^r + (q-1)^{2 r} - (q^2+q+1)^r) ) \\ 
& +   \frac16 q^6 (-6 (q-1)^r + (q-1)^{2 r} + 3 (q^2-1)^r + 2 (q^2+q+1)^r).
 \end{align*}
\end{enumerate}

Adding up all the contributions we get:
  \begin{align*}
e(\bar\cR^{red}_{r,3})= & \frac13 (q^2+q+1)^r (q-1 )^2 q^3 (q+1 ) +(q^2+q+1 ) (2q^{2r}-q^2) (q-1)^{2r} q^r (q+1)^r \\
 &-\frac13 ( q-1)^{2r}(q+1)(q^2+q+1)(3q^{3r}- 3q^{r+2}+ q^3) \\
 =& e(\cR^{red}).
 \end{align*}
From this 
$e(\bar\cR^{irr}_{r,3})= e(\cR^{irr}_{r,3})$ and  $e(\bar\cM^{irr}_{r,3})=e(\cM^{irr}_{r,3})$.

The remaining thing to compute is $e(\bar\cM_{r,3}^{red})$. This is formed by two strata:
 \begin{enumerate}
 \item $\bar M_0=M_0/(\ZZ_3)^r \cong ((\CC^*)^r \x \cM^{irr}_{r,2})/(\ZZ_3)^r$. Hence
$e(\bar M_0)=(q-1)^r e(\bar\cM^{irr}_{r,2})=e(M_0)$.
\item $\bar M_1=M_1/(\ZZ_3)^r \cong (((\CC^*)^r)/(\ZZ_3)^r)/\Sigma_3 \cong (\CC^*)^r/\Sigma_3$.
So $e(\bar M_1)=e(M_1)$.
\end{enumerate}
We get finally $e(\bar\cM^{red}_{r,3})=e(\cM^{red}_{r,3})$. This
concludes the proof of the equality $e(\bar\cM_{r,3})=e(\cM_{r,3})$.

\begin{rem}
There is an arithmetic argument communicated to us by S. Mozgovoy
to prove that $e(\bar\cM_{r,n})=e(\cM_{r,n})$ for $n$ odd. It goes as follows:
find infinitely many primes $p$ such that $p-1$ and $n$
are coprime $($by Dirichlet's theorem on arithmetic progressions$)$; then
$\SL(n,\FF_p) \to \PGL(n,\FF_p)$ is bijective and one gets a bijection between
 corresponding character varieties over $\FF_p$. So the count number
of points of $\cM_{r,n}$ and $\bar\cM_{r,n}$ over $\FF_p$ coincide,
and hence the $E$-polynomials coincide. 

However this argument cannot be used for even $n$. Despite this, the $E$-polynomials for the $\SL(2,\CC)$-character varieties of free groups do equal those of $\PGL(2,\CC)$.  We expect to address the case of $\SL(4,\CC)$ in future work. 
\end{rem}

\end{document}